\documentclass[11pt,leqno]{article}
\usepackage{amsmath,amssymb,amsthm,nicefrac,mathrsfs}
\usepackage[colorlinks]{hyperref}
\usepackage{graphicx}
\usepackage[caption=false]{subfig}
\usepackage{tabu}
\usepackage{bbm}
\usepackage{hyperref}
\usepackage{indentfirst}
\usepackage{mathrsfs}
\usepackage{physics}
\usepackage{lipsum}
\usepackage{verbatim}
\usepackage[a4paper, margin=4cm, footskip=1cm]{geometry}
\usepackage{caption}
\usepackage[T1]{fontenc}
\usepackage{booktabs}
\usepackage{siunitx}
\usepackage{scalerel,stackengine}
\stackMath
\newcommand\reallywidehat[1]{%
	\savestack{\tmpbox}{\stretchto{%
			\scaleto{%
				\scalerel*[\widthof{\ensuremath{#1}}]{\kern-.6pt\bigwedge\kern-.6pt}%
				{\rule[-\textheight/2]{1ex}{\textheight}}
			}{\textheight}%
		}{0.5ex}}%
	\stackon[1pt]{#1}{\tmpbox}%
}

\providecommand{\keywords}[1]
{
	\small	
	\textbf{\textit{Keywords---}} #1
}

\newcommand{\Addresses}{{
		\bigskip
		\footnotesize

	\noindent 	NGARTELBAYE GUERNGAR \\
		\textsc{Department of Mathematics, University of North Alabama,
		Florence, AL 35832}\\
		\textit{E-mail address}: \texttt{nguerngar@una.edu}\\
		\textit{URL}: \texttt{\url{	https://www.researchgate.net/profile/Ngartelbaye_Guerngar}}
	
		\medskip

	\noindent ERKAN NANE \\
		 \textsc{Department of Mathematics and Statistics, Auburn University,
		 	Auburn, AL 36849}\\
		\textit{E-mail address}: \texttt{ezn0001@auburn.edu}\\
		\textit{URL}: \texttt{\url{http://www.auburn.edu/~ezn0001}}
		
		\medskip
		
		\noindent S\"ULEYMAN ULUSOY\\
		\textsc{ Department of Mathematics and Natural Sciences, American University of Ras Al Khaimah, Ras Al Khaimah, UAE}\\
		\textit{E-mail address}: \texttt{suleyman.ulusoy@aurak.ac.ae}\\
		\textit{URL}: \texttt{\url{	https://www.researchgate.net/profile/Suleyman_Ulusoy}}

		\medskip
		\noindent HANS-WERNER VAN WYK \\
		\textsc{Department of Mathematics and Statistics, Auburn University,
			Auburn, AL 36849}\\
		\textit{E-mail address}: \texttt{hzv0008@auburn.edu}\\
		\textit{URL}: \texttt{\url{http://www.auburn.edu/~hzv008}}
		
}}

\newtheorem{theorem}{Theorem}[section]
\newtheorem{definition}[theorem]{Definition}
\newtheorem{lemma}[theorem]{Lemma}

\numberwithin{equation}{section}

\newcommand{\dt}{\, dt}

{%

\begin{enumerate}}%
{\end{enumerate}
}

\def\be{\begin{equation}}
\def\ee{\end{equation}}
\def\ba{\begin{aligned}}
\def\ea{\end{aligned}}
\def\bes{\begin{equation*}}
\def\ees{\end{equation*}}

\def\bc{\begin{cases}}
\def\ec{\end{cases}}
\numberwithin{equation}{section}
\everymath{\displaystyle}

\author{Ngartelbaye Guerngar\\
	University of North Alabama\\
	\and Erkan Nane\\ Auburn University\\ \and Suleyman Ulusoy\footnote{The research of S.U. has been partially supported by BAGEP 2015 award.}\\American University of Ras Al Khaimah\\ \and Hans Werner Van Wyk\\ Auburn University
}
\title{{ A uniqueness determination of the fractional exponents in a three-parameter fractional diffusion}	
	\date{}
}

\begin{document}
\maketitle





\begin{abstract}
\noindent In this article, we consider the space-time Fractional (nonlocal) diffusion equation
$$\partial_t^\beta u(t,x)={\mathtt{L}_D^{\alpha_1,\alpha_2}} u(t,x), \ \ t\geq 0, \ x\in D, $$
where $\partial_t^\beta$ is the Caputo fractional derivative of order $\beta \in (0,1)$ and  the differential operator ${\mathtt{L}_D^{\alpha_1,\alpha_2}}$ is the generator of a L\'evy process, sum of two symmetric independent $\alpha_1-$stable and $\alpha_2-$stable processes and ${D}$ is the open unit interval in $\mathbb{R}$. We consider
a nonlocal inverse problem and show that the fractional exponents $\beta$ and $\alpha_i, \ i=1,2$ are determined uniquely by the data $u(t, 0) = g(t),\
0 < t < T.$ The uniqueness result is a theoretical background for determining
experimentally the order of many anomalous diffusion phenomena, which are important in many fields, including physics and environmental engineering. We also discuss the numerical approximation of the inverse problem as a nonlinear least-squares problem and explore parameter sensitivity through numerical experiments. 

\end{abstract}


\newpage
\section{Introduction}
\indent While the traditional diffusion equation $\partial_t u= \Delta u$ describes a cloud of spreading particles at the macroscopic level, the  space-time fractional diffusion equation $\partial_t^\beta u= -(-\Delta)^{\alpha/2}u$ with $0<\beta<1$ and $0<\alpha<2$ models anomalous diffusions. The fractional derivative in time can be used to describe particle sticking and trapping phenomena. The fractional space derivative models long particle jumps. The combined effect produces a concentration profile with a sharper peak, and heavier tails \cite{CMN,MBSB}. Here the fractional Laplacian $(-\Delta)^{\alpha/2}$ is the infinitesimal
generator of a symmetric $\alpha-$ stable process $X= \Big\{  X_t, \ t\geq 0, \mathbb{P}_x, \ x\in \mathbb{R}^d \Big\}$,  a typical  example of a non-local operator {\color{red}\cite{Sato}}. This process is  a L\'evy process satisfying

$$
\mathbb{E}\Big[ e^{i\xi(X_t-X_0)}\Big]= e^{-t|\xi|^\alpha} \ \ \ \text{for every} \ x, \xi\in \mathbb{R}^d.
$$
{ Here and in what follows, $\mathbb{E}[X]$ is the mathematical expectation of the random variable $X$.}
 In this paper, we consider the equation
\begin{equation}\label{DrPb}
\partial_t^\beta u= -(-\Delta)^{\alpha_1/2}u- (-\Delta)^{\alpha_2/2}u \ \ \text{ with} \ \  0<\beta<1 \ \ \text{ and} \ \  0<\alpha_1\leq \alpha_2<2.
\end{equation}
 Suppose $X$ is a symmetric $\alpha_1-$ stable process and $Y$ is a symmetric $\alpha_2-$stable process, both defined on $\mathbb{R}^d$, and that $X$ and $Y$ are independent. We define the process $Z=X+Y$. Then the infinitesimal generator of $Z$ is  $(-\Delta)^{\alpha_1/2}+ (-\Delta)^{\alpha_2/2}$.
The L\'evy process $Z$ runs on two different scales: on the small spatial scale, the $\alpha_2$ component dominates, while on the large spatial scale the $\alpha_1$ component takes over. Both components play essential roles, and so in general this process can not be regarded as a perturbation of the $\alpha_1-$stable process or of the $\alpha_2-$ stable process. Note that this process can not be obtained from symmetric stable processes through a combination of Girsanov transform and Feynman-Kac transform \cite{ChKmSg}.

\noindent The fractional-time  derivative considered here is the Caputo fractional derivative of order $0<\beta <1$ and is defined as

\begin{equation}\label{CapDer}
\partial_t^\beta q(t)= \frac{\partial^\beta q(t)}{\partial t^\beta}:= \frac{1}{\Gamma(1-\beta)}\int_0^t \frac{\partial q(s)}{\partial s}\frac{ds}{(t-s)^\beta},
\end{equation}
where $\Gamma(\cdot)$ is the Euler's gamma function. For example, $\partial_t^\beta (t^p)= \frac{t^{\beta-p}\Gamma(p+1)}{\Gamma(p+1-\beta)}$ for any $p>0$. This definition of the Caputo fractional derivative is intended to properly handle initial values \cite{Cap,CMN,10}, since its Laplace transform $s^\beta \tilde{q}(s)-s^{\beta -1} q(0)$ incorporates the initial value in the same way the first derivative does. Here, $\tilde{q}(s)=\int_0^\infty e^{-ts} q(t)dt$ represents the usual Laplace transform of the function $q$. \\
It is also well known that, if $q\in C^1(\mathbb{R}^+)$ satisfies $|q'(t)|\leq C t^{\nu -1}$ for some $\nu>0$, then by \eqref{CapDer}, the Caputo derivative of $q$ exists for all $t>0$ and the derivative is continuous in $t>0$ \cite{Kil, 11}.

The following class of functions will play an important role in this article.
\begin{definition}\label{MtgLfl}
	The Generalized (two-parameter) Mittag-Leffler function is defined by:
	\begin{equation}\label{mtglfr}
	E_{\beta,\alpha}(z)=\sum_{k=0}^{\infty}\frac{z^k}{\Gamma(\beta k+\alpha)}, \ \ z\in\mathbb{C}, \ \ \ \Re (\alpha)>0, \ \ \ \Re (\beta)>0,
	\end{equation}
	where  $\Re(\cdot)$ is the real part of a complex number.  When $\alpha=1$, this function reduces to $E_{\beta}(\cdot):=E_{\beta,1}(\cdot).$
\end{definition}

\noindent It is well-known that the Caputo derivative has a continuous spectrum \cite{CMN,11}, with eigenfunctions given in terms of the Mittag-Leffler function. In fact, it is not hard to check that the function $q(t)=E_\beta(-\lambda t^\beta)$ is a solution of the eigenvalue equation
$$
\partial_t^\beta q(t)=-\lambda q(t) \ \ \text{for any} \ \lambda>0.
$$

 For $0<\alpha_1\leq\alpha_2<2,$   $\Big[-(-\Delta)^{\alpha_1/2}- (-\Delta)^{\alpha_2/2}\Big]h$ is defined for
$$
h\in \text{Dom}\Big(-(-\Delta)^{\alpha_1/2}- (-\Delta)^{\alpha_2/2}\Big):= \Big\{ h\in L^2(\mathbb{R}^d; dx): \int_{\mathbb{R}^d} \Big(|\xi|^{\alpha_1}+|\xi|^{\alpha_2} \Big)|\hat{h}(\xi)|^2 d\xi<\infty \Big\}
$$
as the function with Fourier transform

$$
\mathcal{F}\Big[\Big(-(-\Delta)^{\alpha_1/2}- (-\Delta)^{\alpha_2/2}\Big)h(\xi)\Big]= -\Big(|\xi|^{\alpha_1}+|\xi|^{\alpha_2} \Big)|\hat{h}(\xi)|^2.
$$
Here, $\mathcal{F}(h)= \hat{h}$ represents the usual Fourier transform of the function $h$.

The main purpose of this article is to establish the determination of the unique
exponents $\beta$ and $\alpha_i, \ i=1,2$  in the {time and space fractional} derivatives by means of the observed data (also called additional condition) $u(t,0)= g(t), \ \ 0<t<T$, where $g(t)\not\equiv 0.$ 

Many works have been done recently in inverse problems \cite{CNYY, JR, Mis2, Mis1, Mis5, Mis6, TrPeSy, TarSoy,  Mis3, Mis4, ZX}. While most of these works have been dedicated to fractional derivatives only in the time variable \cite{CNYY, JR, Mis2, Mis1, Mis5, Mis6,  Mis3, Mis4, ZX}, space-time fractional derivatives  were considered in \cite{TrPeSy, TarSoy}, similarly as in this article. However, a substantial difference is that our work considers diffusion equation involving two independent processes.

The rest of this  article is organized as follows: in the next section we provide a review of
main properties of the direct problem and introduce the inverse problem. Section
3 is devoted to both the statement and the proof of the main result of this paper. Section 4 concerns numerical approximations to our problem. A conclusion in section 5 ends this paper. Throughout this article, the letter $c$, in upper or lower case, with or without a subscript, denotes a constant whose value is not of interest in this article and may stay the same or change from line to line. For simplicity, we will fix $d=1$ in the remainder of this paper. The following notation will be used in the sequel:  $D=(-1,1)$; for $a,b\in \mathbb{R},$ $a\wedge b:= \min(a,b)$; for any two positive functions $p$ and $q$, $p\asymp q$ means that there is a positive constant $c\geq 1$ so that $c^{-1}q\leq p\leq c q$ on their common domain of definition. For a given set $A\subset \mathbb{R},$ $A^C=\mathbb{R}-A.$

\section{Analysis of the direct problem and formulation of the inverse problem}
We start by considering the direct problem.  The equation we are interested in  reads as
\begin{equation}\label{eq1}
\begin{cases}
  &\partial_t^\beta u(t, x) = -(-\Delta)^{\alpha_1/2}u(t,x) - (-\Delta)^{\alpha_2/2}u(t,x), ~t\geq 0, x \in D,\\
  & u(t, x) =  0,
  \ x\in D^C, ~0 < t < T,\\
  & u(0, x) = f(x), ~x\in D.
\end{cases}
\end{equation}
Here $T > 0$ is a final time and $f$ is a given function. 

\noindent We define the operator ${\mathtt{L}_D^{\alpha_1,\alpha_2}}:= -( -\Delta)^{\alpha_1/2} -(-\Delta)^{\alpha_2/2}$  for $0<\alpha_1\leq \alpha_2<2$ on ${ D}$.
\begin{definition}[\cite{CMN}]
A function $u(t,x)$ is said to be a weak solution of \eqref{eq1} if the following conditions are satisfied:
\begin{equation}
\begin{split}
&u(t,\cdot) \in \mathscr{F} \ \ \text{for each} \ t>0,\\
&\lim\limits_{t\downarrow 0} u(t,x)= f(x) \ \ a.e, \\
& \partial_t^\beta u(t,x)= {\mathtt{L}_D^{\alpha_1,\alpha_2}} u(t,x) \ \ \text{in the distributional sense, i.e }
\end{split}
\end{equation}
$$
\int_{\mathbb{R}}\Bigg(\int_0^\infty u(t,x)\partial_t^\beta \psi(t) \Bigg)\phi(x)dx= \int_0^\infty \varepsilon^D(u(t,.),\phi)\psi(t) dt
$$
for every $ \psi\in C_0^1(\mathbb{R}^+) \ \text{and} \ \phi\in C_0^2(D)$.
\noindent Here, $\mathscr{F}$ is the $\sqrt{\epsilon_1}-$ completion of the space $C_0^{\infty}(D)$ of smooth functions with compact support in $D$, where $$\epsilon_1(u,u)= \epsilon(u,u)+\int_{\mathbb{R}} u^2(x)dx,$$ $$\epsilon(u,v)= \varepsilon^D(u,v) \ \ \text{for} \ u,v\in \mathscr{F},$$
and
$$
\varepsilon^D(u,v)= \frac{1}{2}\int_{D^2} \Big(u(x)-u(y)\Big)\Big(v(x)-v(y)\Big) \Bigg(\frac{\mathcal{A}( -{\alpha_2})}{|x-y|^{1+\alpha_2}} + \frac{b}{|x-y|^{1+\alpha_1}}\Bigg)\ dxdy,
$$
where $\mathcal{A( -\alpha)}= \alpha2^{\alpha-1}\pi^{-1/2} \Gamma\big((1+\alpha)/2\big)\Gamma(1-\alpha/2)^{-1} $ and $b\in \mathbb{R}$, for $u,v\in \mathscr{F}$  \cite{ChKmSgVk}.
\end{definition}
\noindent $\varepsilon^D(u,v)$ comes from variational formulation and symmetry,  and
$$
\mathscr{F}:=\Bigg\{u\in L^2(D; dx): \int_{D^2}\Big(u(x)-u(y)\Big)^2 \Bigg(\frac{\mathcal{A}( -{\alpha_2})}{|x-y|^{1+\alpha_2}} + \frac{b}{|x-y|^{1+\alpha_1}}\Bigg)\ dxdy<\infty \Bigg\}.
$$

Following \cite{CMN}, a weak solution of Problem \eqref{eq1} is given by the following formula
\begin{equation}\label{PwExp}
\begin{split}
u(t,x)= &\int_0^\infty \mathbb{E}_x\Big[f(Z_s); s<\tau_D \Big]f_t(s)ds\\
=& \int_0^\infty\Bigg( \sum_{n=1}^\infty e^{-s\mu_n}\langle f, \varphi_n\rangle \varphi_n(x)\Bigg) f_t(s)ds\\
=& \sum_{n=1}^{\infty} E_\beta(-\mu_nt^\beta)\langle f, \varphi_n\rangle \varphi_n(x),
\end{split}
\end{equation}
where $f_t(.)$ is defined in \cite[(2.1)]{CMN}, $\tau_D$ is defined later in \eqref{tau}, $(\mu_n)_{n\geq 1} $ is a sequence of positive numbers satisfying $0<\mu_1\leq \mu_2\leq \cdots$ and $(\varphi_n)_{n\geq 1}$ is an orthonormal basis of $L^2(D)$,  satisfying the following system of equations

\begin{equation}\label{Eigpair}
\begin{cases}
&{\mathtt{L}_D^{\alpha_1,\alpha_2}}\varphi_n= -\mu_n\varphi_n  \ \ \text{on} \ \ D\\
& \varphi_n= 0 \ \ \text{on} \ \ D^C.

\end{cases}
\end{equation}

Hence, any function $f\in L^2(D;dx)$ has the representation
\begin{equation}\label{fRep}
f(x)=\sum_{n=1}^\infty \langle f, \varphi_n \rangle \varphi_n(x).
\end{equation}
Using the spectral representation, one has
\begin{equation}\label{DomLD}
\text{Dom}\big({\mathtt{L}_D^{\alpha_1,\alpha_2}}\big)= \Big\{f\in L^2(D): {\big\|{\mathtt{L}_D^{\alpha_1,\alpha_2}}f\big\|}_{L^2(D)}^2 = \sum_{n=1}^\infty \mu_n^2\langle f, \varphi_n\rangle^2<\infty\Big\}
\end{equation}
and
$$
{\mathtt{L}_D^{\alpha_1,\alpha_2}}f(x)= -\sum_{n=1}^\infty \mu_n \langle f, \varphi_n \rangle \varphi_n(x).
$$
For any real-valued function $\phi: \mathbb{R}\rightarrow \mathbb{R},$ one can also define the operator $\phi({\mathtt{L}_D^{\alpha_1,\alpha_2}})$ as follows:
\begin{equation}\label{DomPhiLD}
\text{Dom}\Big(\phi\big({\mathtt{L}_D^{\alpha_1,\alpha_2}}\big)\Big)= \Big\{f\in L^2(D): {\Big\|\phi\big({\mathtt{L}_D^{\alpha_1,\alpha_2}}\big)f\Big\|}_{L^2(D)}^2 = \sum_{n=1}^\infty \phi(\mu_n)^2\langle f, \varphi_n\rangle^2<\infty\Big\}
\end{equation}
 and

 \begin{equation}\label{PhiLd}
 \phi\big({\mathtt{L}_D^{\alpha_1,\alpha_2}}\big)f= \sum_{n=1}^\infty \phi(\mu_n) \langle f, \varphi_n \rangle \varphi_n.
 \end{equation}

 \noindent For the remainder of this article, we will use $\phi(t)= t^k$ for some $k>0.$
 For technical reasons (cf. proof of main Theorem), we also restrict $f$ to the class of functions satisfying
 \begin{equation}\label{Condf}
 \langle f, \varphi_n\rangle > 0, \ \ n\geq 1 \ \ \Big( \text{or} \ \  \langle f, \varphi_n\rangle < 0, \ n\geq 1\Big).
 \end{equation}
 The following lemma indicates an important property of the Mittag-Leffler function. It will be used frequently in the sequel.
 \begin{lemma}{\color{red}\cite{11}}
 For each $0<\alpha<2$ and $\pi\alpha/2<\mu<\min(\pi, \pi\alpha)$, there exists a constant $C_0>0$ such that
 \begin{equation}\label{Mitg}
 \big|E_\beta(z)\big|\leq \frac{C_0}{1+|z|}, \ \ \mu\leq |\arg(z)|\leq \pi.
 \end{equation}
 \end{lemma}

\begin{theorem}
The eigenvalues of the spectral problem for the one-dimensional
double fractional Laplace operator, i.e $(-\Delta)^{\alpha_1}u(x) + (-\Delta)^{\alpha_2}u(x)= \mu_n u(x)$ in the interval $D\subset \mathbb{R}$ satisfy the following bounds

\begin{equation}\label{AstEigv}
c_1(n^{\alpha_1}+n^{\alpha_2})	\leq \mu_n\leq c_2  (n^{\alpha_1}+n^{\alpha_2}), \ \ \text{for all} \ n\geq 1 \ \ \text{and} \ \ c_1, c_2>0.
\end{equation}
\end{theorem}
\begin{proof}
This follows easily from \cite[Theorem 4.4]{ChSg} by taking $\phi(s)= s^{\alpha_1}+ s^{\alpha_2}. $
\end{proof}
For the existence of a solution to \eqref{eq1}, we now show that the series given in  \eqref{PwExp} is uniformly convergent for $(t,x)\in (0,T]\times D$. To this aim, we use the following Lemma giving bounds for the eigenvalues and eigenfunctions:

\begin{lemma}\label{lem:eigenfunction-bound}
	Suppose that the initial value $f$ in \eqref{eq1} is such that $ f\in \text{Dom}\Big({\big(\mathtt{L}_D^{\alpha_1,\alpha_2}\big)}^k\Big)$ for  $k>-1+\frac{3}{2\alpha_2}.$ Let  $(\mu_n, \varphi_n)$ be the eigenpair from \eqref{Eigpair},  then
\begin{equation}\label{EigBd}
\begin{split}
&\big|\langle f,\varphi_n \rangle\big|\leq \sqrt{M}\mu_n^{-k}\\
& \big|\varphi_n(x) \big|\leq c_3 \Big(\mu_n^{1/2\alpha_1}\wedge \mu_n^{1/2\alpha_2}\Big),
\end{split}
\end{equation}
where $$
M:= \sum_{n=1}^\infty \mu_n^{2k}\langle f, \varphi_n\rangle^2<\infty \ \ \text{and} \ \ c_3>0.
$$

\end{lemma}
\begin{proof}
The first bound in \eqref{EigBd} follows directly from the definition of $M$. So we only show the second bound.\\
Recall that the fundamental solution $p(t,x,y)$, also referred to as the heat kernel of ${\color{red}\mathtt{L}^{\alpha_1,\alpha_2}}$,  is the unique solution to
\begin{equation}
\partial_t u= {\mathtt{L}^{\alpha_1,\alpha_2}} u.
\end{equation}
It represents  the transition density function of $Z$. Denote the first exit time of the process $Z$ by
\begin{equation}\label{tau}
\tau_D:=\inf\{ t\geq 0: Z_t\notin D\}.
\end{equation}
Let $Z^D$ denote the process $Z$ "killed" upon exiting $D$, i.e

\begin{equation}
Z^D_t:= \begin{cases}
&Z_t , \ \ \ \ \  t<\tau_D \\
&\partial,  \ \ \ \ t\geq \tau_D
\end{cases}
\end{equation}
Here, $\partial$ is a cemetery  point added to $D$. Throughout this paper, we use the convention that any real-valued function $f$ can be extended by taking $f(\partial)=0.$
Then $Z^D$ has a jointly continuous transition density function $p_D(t,x,y).$ Moreover, by the strong Markov property of $Z$, one has for $t>0$ and $x,y\in D$,
\begin{equation}\label{StMkvP}
p_D(t,x,y)= p(t,x,y)-E\big[p(t-\tau_D, X_{\tau_D},y); t<\tau_D \big]\leq p(t,x,y).
\end{equation}
By \cite[(1.4)]{ChKmSg},
\begin{equation}\label{AsympHK}
p(t,x,y)\asymp\Big( t^{-1/\alpha_1}\wedge t^{-1/\alpha_2} \Big)\wedge\Bigg( \frac{t}{|x-y|^{1+\alpha_1}}+\frac{t}{|x-y|^{1+\alpha_2}}\Bigg).
\end{equation}
In particular, one has $\sup\limits_{x\in D}\int_D p(t,x,y)^2 dy<\infty$ for all $t>0.$
Denote by $\Big\{ p_t^D, t\geq 0\Big\}$ the transition semigroup of $Z^D$, i.e

$$
p_t^Df(x)= \int_D p_D(t,x,y)f(y)dy.
$$
It is well know ( cf. \cite{FukShDa}) that $u(t,x)= p_t^Df(x)$ is the unique weak solution  to

\begin{equation*}
\partial_t u= {\mathtt{L}_D^{\alpha_1,\alpha_2}} u
\end{equation*}
with initial condition $u(0,x)=f(x)$ on the Hilbert space $L^2(D;dx)$.
Therefore,  for each $t > 0$, $p_t^D$ is a Hilbert-Schmidt operator in $L^2(D; dx)$ so it is compact \cite{CMN}. Consequently, for the eigenpair defined in \eqref{Eigpair}, we have $p_t^D \varphi_n=e^{-\mu_n t}\varphi_n$ in $L^2(D;dx)$ for $n\geq 1$ and $t>0.$ Combining this with \eqref{fRep}, it follows that

\begin{equation*}
p_t^Df(x)= \sum_{n=1}^\infty \langle f,\varphi_n\rangle p_t^D\varphi_n=  \sum_{n=1}^\infty e^{-\mu_nt}\langle f,\varphi_n\rangle \varphi_n.
\end{equation*}

In particular, the transition density $p_D(t,x,y)$ is given by

\begin{equation}\label{DchKrn}
p_D(t,x,y)=  \sum_{n=1}^\infty e^{-\mu_nt}\varphi_n(x)\varphi_n(y).
\end{equation}

Next,

\begin{equation*}
e^{-\mu_n t}\big|\varphi_n(x)\big|^2\leq \sum_{m=1}^\infty e^{-\mu_m t}\big|\varphi_m(x)\big|^2= p_D(t,x,x)\leq p(t,x,x)\leq C_1\Big( t^{-1/\alpha_1}\wedge t^{-1/\alpha_2} \Big).
\end{equation*}
Hence, taking the square root of both sides, we get

\begin{equation}
\big|\varphi_n(x)\big|\leq C_2e^{\mu_n t/2} \sqrt{t^{-1/\alpha_1}\wedge t^{-1/\alpha_2} }.
\end{equation}
Finally, taking $t=\mu_n^{-1}$ concludes the proof.\\
\end{proof}

\noindent With everything set, we can now proceed to show the uniform convergence of the series given in \eqref{PwExp}. In fact, using \eqref{Mitg}, \eqref{AstEigv} and \eqref{EigBd}, we have

\begin{equation}
\begin{split}
\sum_{n=1}^{\infty} \max\limits_{x\in D}\big|E_\beta(-\mu_nt^\beta)\langle f, \varphi_n\rangle \varphi_n(x)\big|\leq & \sqrt{M}C\sum_{n=1}^\infty \frac{1}{1+|\mu_n t^\beta|}\mu_n^{-k}\Big(\mu_n^{1/2\alpha_1}\wedge \mu_n^{1/2\alpha_2}\Big)\\
\leq & C\sum_{n=1}^\infty n^{(-k-1)\alpha_2+1/2}<\infty
\end{split}
\end{equation}
by our choice of $k$ in Lemma \ref{lem:eigenfunction-bound}. This shows that the series in \eqref{PwExp} is uniformly convergent.\\

We are now  ready   to state and prove our main result.
\section{Statement and proof of the main result}
We open this section straight with  our main result. We then provide its proof.
\begin{theorem}\label{thm:parameters_determine_path_uniquely}
{ Suppose $0<\gamma<1$ and $0<\eta_1\leq \eta_2<2$.} Let $u$ be the weak solution of \eqref{eq1} and let $v$  be the weak solution of the following problem
\begin{equation}\label{UniqP}
\begin{cases}
&\partial^\gamma_t v(t,x)= -(-\Delta)^{\eta_1/2}v(t,x)- (-\Delta)^{\eta_2/2}v(t,x), \ \ t\geq 0, x\in D,\\
& v(t, x) =  0,  \ x\in D^C, ~0 < t < T,\\
& v(0, x) = f(x), \ x\in D.
\end{cases}
\end{equation}
If $u(t,0)=v(t,0), \ 0<t<T$ and  \eqref{Condf} holds, then
$$\beta=\gamma\ \ \text{and} \ \alpha_i= \eta_i, \ i=1,2.$$
\end{theorem}
\begin{proof}
	The proof follows a similar argument as in \cite{TarSoy}.
Using the explicit formula \eqref{PwExp}, the weak solutions $u$ and $v$ can be written as
\begin{equation}\label{PwExpU}
u(t,x)= \sum_{n=1}^{\infty} E_\beta(-\mu_nt^\beta)\langle f, \varphi_n\rangle \varphi_n(x)
\end{equation}
and
\begin{equation}\label{PwExpV}
v(t,x)= \sum_{n=1}^{\infty} E_\gamma(-\lambda_nt^\gamma)\langle f, \psi_n\rangle \psi_n(x),
\end{equation}
where the eigenpairs $\Big(\mu_n, \varphi_n\Big)$ and $\Big( \lambda_n, \psi_n\Big)$ satisfy

\begin{equation*}
\begin{cases}
&{\mathtt{L}_D^{\alpha_1,\alpha_2}}\varphi_n= -\mu_n\varphi_n  \ \ \text{on} \ \ D.\\
& \varphi_n= 0 \ \ \text{on} \ \ D^C

\end{cases}
\end{equation*}

and

\begin{equation*}
\begin{cases}
&\mathtt{L}_D^{\eta_1,\eta_2}\psi_n= -\lambda_n\psi_n  \ \ \text{on} \ \ D.\\
& \psi_n= 0 \ \ \text{on} \ \ D^C,
\end{cases}
\end{equation*}
Without loss of generality, we can normalize the eigenfunctions such that $\varphi_n(0)=\psi_n(0)=1$ for all $n\geq 1.$ This implies that

\begin{equation}\label{EqWSol}
\sum_{n=1}^{\infty} E_\beta(-\mu_nt^\beta)\langle f, \varphi_n\rangle= \sum_{n=1}^{\infty} E_\gamma(-\lambda_nt^\gamma)\langle f, \psi_n\rangle
\end{equation}
if we assume that $u(t,0)= v(t,0)$.\\
Next, we use the following asymptotic property of the Mittag-Leffler function \cite{Kil, 11}
\begin{equation}\label{MitgAspt}
E_l(-t)= \frac{1}{t\Gamma(1-l)}+O(|t|^{-2}), \  \ 0<l<1.
\end{equation}
Combining \eqref{AstEigv} and \eqref{MitgAspt}, we get

\begin{equation}\label{MitgAsymp}
\Bigg|E_\beta(-\mu_n t^\beta)- \frac{1}{\Gamma(1-\beta)}\frac{1}{\mu_n t^\beta}\Bigg|\leq Ct^{-2\beta}.
\end{equation}
By adding and subtracting the term $\frac{1}{\Gamma(1-\beta)}\frac{1}{\mu_n t^\beta}$ in the left side term in \eqref{EqWSol}, we get the following asymptotic equation

\begin{equation}\label{AsyEqWSolU}
\begin{split}
\sum_{n=1}^{\infty} E_\beta(-\mu_nt^\beta)\langle f, \varphi_n\rangle
&= \sum_{n=1}^{\infty}\langle f, \varphi_n\rangle\Bigg[\frac{1}{\Gamma(1-\beta)}\frac{1}{\mu_n t^\beta}+E_\beta(-\mu_n t^\beta)- \frac{1}{\Gamma(1-\beta)}\frac{1}{\mu_n t^\beta} \Bigg]\\
&= \sum_{n=1}^{\infty}\langle f, \varphi_n\rangle\frac{1}{\Gamma(1-\beta)}\frac{1}{\mu_n t^\beta} +O(|t|^{-2\beta}).
\end{split}
\end{equation}
Similarly,
\begin{equation}\label{AsyEqWSolV}
\begin{split}
\sum_{n=1}^{\infty} E_\gamma(-\lambda_nt^\gamma)\langle f, \psi_n\rangle
&= \sum_{n=1}^{\infty}\langle f, \psi_n\rangle\Bigg[\frac{1}{\Gamma(1-\gamma)}\frac{1}{\lambda_n t^\gamma}+E_\gamma(-\lambda_n t^\gamma)- \frac{1}{\Gamma(1-\gamma)}\frac{1}{\lambda_n t^\gamma} \Bigg]\\
&= \sum_{n=1}^{\infty}\langle f, \psi_n\rangle\frac{1}{\Gamma(1-\gamma)}\frac{1}{\lambda_n t^\gamma} +O(|t|^{-2\gamma}).
\end{split}
\end{equation}
Now combining \eqref{EqWSol}, \eqref{AsyEqWSolU} and \eqref{AsyEqWSolV}, we get, as $t\rightarrow\infty$

\begin{equation}\label{AsymEqWSol}
\sum_{n=1}^{\infty}\langle f, \varphi_n\rangle\frac{1}{\Gamma(1-\beta)}\frac{1}{\mu_n t^\beta} +O(|t|^{-2\beta})= \sum_{n=1}^{\infty}\langle f, \psi_n\rangle\frac{1}{\Gamma(1-\gamma)}\frac{1}{\lambda_n t^\gamma} +O(|t|^{-2\gamma}).
\end{equation}
Now assume, for example, that $\beta>\gamma$. Then multiply \eqref{AsymEqWSol} by $t^{\gamma}$ to get
\begin{equation}\label{FinEqWsol}
-t^{\gamma-\beta}\sum_{n=1}^{\infty}\langle f, \varphi_n\rangle\frac{1}{\Gamma(1-\beta)}\frac{1}{\mu_n } +O(|t|^{\gamma-2\beta})+ \sum_{n=1}^{\infty}\langle f, \psi_n\rangle\frac{1}{\Gamma(1-\gamma)}\frac{1}{\lambda_n } +O(|t|^{-\gamma})=0.
\end{equation}
Letting $t\rightarrow\infty$ in \eqref{FinEqWsol} yields

\begin{equation}\label{Contrd}
\sum_{n=1}^{\infty}\langle f, \psi_n\rangle\frac{1}{\Gamma(1-\gamma)}\frac{1}{\lambda_n }=0: \ \ \text{a contradiction to \eqref{Condf}!}
\end{equation}
Similarly, assuming $\gamma>\beta$ also leads to a contradiction. Thus $\beta=\gamma.$

We now prove the second part of the Theorem, i.e $\alpha_i=\eta_i, \ i=1,2.$ To this aim, we will show that $\mu_n=\lambda_n$ for all $n\geq 1.$

Since $\beta=\gamma$, Equation \eqref{EqWSol} becomes

\begin{equation}\label{EqWSolStm}
\sum_{n=1}^{\infty} E_\beta(-\mu_nt^\beta)\langle f, \varphi_n\rangle= \sum_{n=1}^{\infty} E_\beta(-\lambda_nt^\beta)\langle f, \psi_n\rangle.
\end{equation}
Taking the Laplace transform of $E_\beta(-\mu_nt^\beta)$ yields

\begin{equation}\label{LapTrf}
\int_0^\infty e^{-zt}E_\beta(-\mu_n t^\beta) dt= \frac{z^{\beta-1}}{z^\beta+\mu_n}, \ \ \Re z>0.
\end{equation}
Furthermore, taking the Laplace transform of the Mittag-Leffler function term by term,  we get

\begin{equation}\label{MitgLefTrmbT}
\int_0^\infty e^{-zt}E_\beta(-\mu_n t^\beta) dt= \frac{z^{\beta-1}}{z^\beta+\mu_n}, \ \ \Re z>\mu_n^{1/\beta}.
\end{equation}
It follows that $\sup\limits_{t\geq 0}\big|E_\beta(-\mu_nt^\beta)\big|<\infty$ by \eqref{Mitg}. This implies that $\int_0^\infty e^{-zt}E_\beta(-\mu_n t^\beta) dt$ is analytic in the domain $\Re z>\mu_n^{1/\beta}.$ Then by analytic continuity, $\int_0^\infty e^{-zt}E_\beta(-\mu_n t^\beta) dt$ is analytic in the domain $\Re z>0.$\\
Using \eqref{Mitg}, \eqref{AstEigv}, \eqref{EigBd} and Lebesgue's convergence Theorem, we get that $$e^{-t\Re z} t^\beta \ \ \text{is integrable for } \ \ t\in(0,\infty) \ \ \text{with fixed} \ z \ \ \text{such that} \ \Re z>0$$ and
\begin{align*}
\Big|e^{-t\Re z} \sum_{n=1}^{\infty} E_\beta(-\mu_nt^\beta)\langle f, \varphi_n\rangle\Big|\leq& C_0 e^{-t\Re z}\Bigg(\sum_{n=1}^\infty\langle f, \varphi_n\rangle \frac{1}{\mu_nt^{\beta}} \Bigg)\\
\leq & C_0^{'} e^{-t\Re z} t^{-\beta}\sum_{n=1}^{\infty} n^{-\alpha_2(k+1)}<\infty
\end{align*}
by the choice of $k$ in \eqref{EigBd}.\\
Next, for $\Re z>0$, we have

\begin{equation}\label{LpTrEMu}
\int_0^\infty e^{-t\Re z} \sum_{n=1}^{\infty} E_\beta(-\mu_nt^\beta)\langle f, \varphi_n\rangle dt= \sum_{n=1}^{\infty}\langle f, \varphi_n\rangle \frac{z^{\beta-1}}{z^\beta+\mu_n}.
\end{equation}
Similarly,

\begin{equation}\label{LpTrELbd}
\int_0^\infty e^{-t\Re z} \sum_{n=1}^{\infty} E_\beta(-\lambda_nt^\beta)\langle f, \psi_n\rangle dt= \sum_{n=1}^{\infty}\langle f, \psi_n\rangle \frac{z^{\beta-1}}{z^\beta+\lambda_n}.
\end{equation}
This means, by \eqref{EqWSolStm}, \eqref{LpTrEMu} and \eqref{LpTrELbd},

\begin{equation}\label{LstEq}
\sum_{n=1}^\infty \frac{\langle f, \varphi_n\rangle}{\rho+\mu_n}= \sum_{n=1}^\infty \frac{\langle f, \psi_n\rangle}{\rho+\lambda_n}, \ \ \Re \rho >0.
\end{equation}
Since we can continue analytically (in $\rho$) both series in \eqref{LstEq}, this equality actually holds for $\rho\in \mathbb{C}-\Big({\{ \mu_n}\}_{n\geq 1}\cup {\{ \lambda_n}\}_{n\geq 1} \Big)$. \\
We are now ready to show that $\mu_n=\lambda_n$ for all $n\geq 1$. We proceed by induction:\\

Without loss of generally, assume $\mu_1<\lambda_1$. Thus we can find a suitable disk containing $-\mu_1$ but not ${\{ -\mu_n\}}_{n\geq 2}\cup{\{ -\lambda_n\}}_{n\geq 1}. $ Then integrating \eqref{LstEq} over this disk, by the Cauchy's integral formula, we get
$$
2\pi i\langle f, \varphi_1\rangle=0: \ \ \text{this is a clear contradiction to } \ \eqref{Condf}. $$
This means that $ \mu_1=\lambda_1$ since the reverse inequality would also lead to a contradiction.

A similar argument yields $\mu_2=\lambda_2$.
Inductively, we deduce that
\begin{equation}\label{EqEigv}
\mu_n=\lambda_n \ \ \text{ for all}  \  n\geq 1.
\end{equation}
This also means that
\begin{equation}\label{AsymEigvAl}
c_1(n^{\alpha_1}+n^{\alpha_2})\leq \mu_n\leq c_2(n^{\alpha_1}+n^{\alpha_2})
\end{equation}
and

\begin{equation}\label{AsymEigvEt}
c_3(n^{\eta_1}+n^{\eta_2})\leq \mu_n\leq c_4(n^{\eta_1}+n^{\eta_2}), \ \ \text{where} \ \ c_i>0, i=1,2,3,4.
\end{equation}
Assume for example that $\alpha_2<\eta_2$, then combining \eqref{AsymEigvAl} and \eqref{AsymEigvEt} yields
$$
c_3^{'}n^{\eta_2}\leq \mu_n\leq c_2^{'}n^{\alpha_2}, \ \ \text{for all} \ n\geq 1: \ \text{a contradiction!}
$$
Therefore $\alpha_2=\eta_2$ since the reverse inequality would also lead to a contradiction.\\
Similarly, assuming $\alpha_1>\eta_1$ and combining \eqref{AsymEigvAl} and \eqref{AsymEigvEt}
gives
$$
c_1(n^{\alpha_1}+n^{\alpha_2})\leq c_4(n^{\eta_1}+n^{\alpha_2}), \ \ \text{for all} \ \ n\geq 1: \ \text{a contradiction!}
$$
Thus $\alpha_1=\eta_1$ and this concludes the proof.
\end{proof}

{

\section{Numerical Approximation}

In this section we discuss the numerical identification of parameters $\alpha_1, \alpha_2$, and $\beta$ from observations $g(t)$,  of $u(t,0)$ for $0 \leq t \leq T$, where $u(t,x)$ solves Equation \eqref{eq1}. We focus here on a least squares formulation of the parameter identification problem. Let $\theta=(\alpha_1,\alpha_2,\beta)$ denote the parameter vector and let
\[
\Theta := \{(\alpha_1,\alpha_2, \beta): 0<\alpha_1,\alpha_2<2, 0<\beta<1\}
\]
be the set of admissible parameters. Since the double fractional Laplacian $\mathtt{L}^{\alpha_1,\alpha_2}$ is symmetric in $\alpha_1$ and $\alpha_2$, we do not enforce the constraint $\alpha_1<\alpha_2$ in $\Theta$. Further, to indicate the dependence of the solution $u$ of Equation \eqref{eq1} on $\theta$, we write $u(t,x;\theta)$. The parameter identification problem can now be expressed as the least-squares optimization problem 
\begin{equation}\label{eq:parid_problem}
\min_{\theta \in \Theta} J(\theta) = \int_0^T |u(t,0;\theta)-g(t)|^2 \dt + \frac{\lambda}{2}\|\theta\|_2^2,
\end{equation}
where $\|\cdot\|_2$ denotes the Euclidean norm and $\lambda >0$ is a suitably chosen Tikonov regularization parameter. The regularization term improves the stability of the minimizer $\theta_\lambda$ in the presence of measurement noise at the cost of biasing the estimate. Heuristic methods are typically used to choose the parameter $\lambda$ that balances these two errors, the most well-known of which is the Morozov discrepancy principle \cite{Morozov1984}. In our computations we approximate $J(\theta)$ by a quadrature rule with nodes $0\leq t_1 < \hdots < t_m \leq T$ and weights $w_1,\hdots,w_m$, resulting in the approximation
\begin{equation*}
J(\theta) \approx \frac{1}{2}\sum_{i=1}^m w_i \left[u(t_i,0;\theta)-g(t_i)\right]^2 + \frac{\lambda}{2} \|\theta\|_2^2.
\end{equation*}
Since the solution of Equation \eqref{eq1} cannot be expressed in closed form, it must be approximated numerically. Moreover, the parameter vector $\theta$ is unknown \emph{a priori}, and hence care must be taken to ensure approximations whose accuracy is uniform over the parameter set $\Theta$. We base our computations on the weak form of $u$ given by Equation \eqref{PwExp}. Its evaluation requires that of the Mittag-Leffler function, as well as the eigenfunctions of the operator $\mathtt L^{\alpha_1,\alpha_2}$. We use the numerical method developed in \cite{Gorenflo2002} to evaluate the Mittag-Leffler function and approximate the eigenfunctions and eigenvalues of $\mathtt L^{\alpha_1,\alpha_2}$ by diagonalizing the discretized operator $\hat {\mathtt L}^{\alpha_1,\alpha_2}$, based on the finite difference method developed in \cite{Duo2018}, which is second order accurate irrespective of $\alpha_1$ and $\alpha_2$. The numerical approximation $\hat{u}$ of $u$ then takes the form
\[
\hat{u}(t_i,x;\theta) = \sum_{n=1}^N \hat{E}_{\beta}(-\hat{\mu}_n  t_i^\beta) \langle f, \hat{\varphi}_n\rangle \hat{\varphi}_n(x),
\]
where $\hat{E}_{\beta}$ denotes the approximation of $E_\beta$, and $\hat{\mu}_n$ and $\hat{\varphi}_n$, $n=1,...,N$, denote the approximate eigenvalues and eigenfunctions obtained via finite differencing. 

Problem \eqref{eq:parid_problem} and its approximation are box-constrained nonlinear least squares problems that can be solved using well-known least squares algorithms, such as the Levenberg-Marquardt method or the Gauss-Newton method (see \cite{GueNanOy}). In this section, we investigate the difficulties in solving Problem \eqref{eq:parid_problem} that arise from the lack of  parameter identifiability. To be sure, Theorem \ref{thm:parameters_determine_path_uniquely} establishes that the parameters $\alpha_1, \alpha_2$ and $\beta$ \emph{uniquely} determine the trajectory $u(t,0;\theta)$. In the following numerical experiment we show however that it is possible for the cost functional $J(\theta)$ to be quite insensitive to the parameters $\alpha_1$ and $\alpha_2$ near the unique minimizer. This leads to difficulties for gradient-based optimization algorithms, resulting in slow convergence or premature termination. 

In our numerical experiment to investigate parameter identifiability, we specify a reference parameter $\theta^* = (\alpha_1^*, \alpha_2^*,\beta^*)=(0.5,1.5,0.7)$ and compute the associated model output $u(t,x;\theta^*)$ for $x\in D=(-1,1)$ and $0<t<1$, which we use as a reference trajectory, i.e. $g(t)=u(t,0;\theta^*)$, $0<t<1$.  To eliminate any additional source of error, we consider only noise-free observations and set the regularization parameter $\lambda=0$. The initial condition $f(x)=(1-x^2)^\frac{7}{2}$, $x\in (0,1)$ is sufficiently smooth to ensure a uniform second order spatial approximation. Our numerical discretization $\hat{u}$ of $u$ has a spatial resolution of $\Delta x=10^{-2}$, a temporal resolution of $\Delta t=10^{-2}$, and full spectral resolution, i.e. all terms in the spectral expansion are retained. We use the composite Simpson rule to evaluate the time-integral of the squared residual.   

\begin{figure}[tbhp!]
\centering
\includegraphics[scale=0.75]{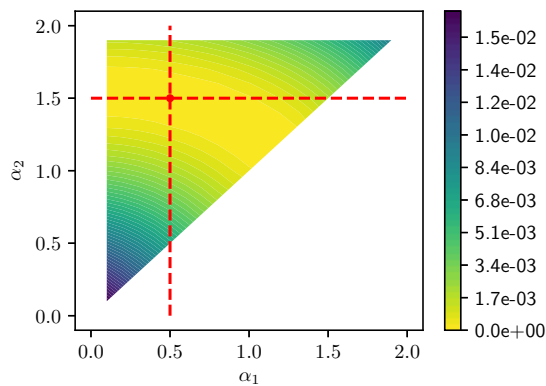}
\caption{Contour plot of $J(\alpha_1,\alpha_2,0.7)$.}
\label{fig:ex01_contour_alphas}
\end{figure}

We first fix $\beta=0.7$ and examine the bivariate cost function $J(\alpha_1,\alpha_2,0.7)$, shown in Figure \ref{fig:ex01_contour_alphas}. The contour plot shows clearly that there is a strip containing $(\alpha_1^*,\alpha_2^*)$ in which the cost $J$ (i) attains comparatively small values and (ii) show very little variation. This behavior is also evident in the cross-sections plotted in Figure \ref{fig:ex01_cross_sections}. In particular, the cross-section of $\alpha_1$ when $\alpha_2=\alpha_2^*$ shows a flat cost functional in a region of $\alpha_1^*$.

\begin{figure}[tbhp!]
\centering
\subfloat[Cross section $J\left(\frac{1}{2},\alpha_2,0.7\right)$]{\includegraphics[scale=0.75]{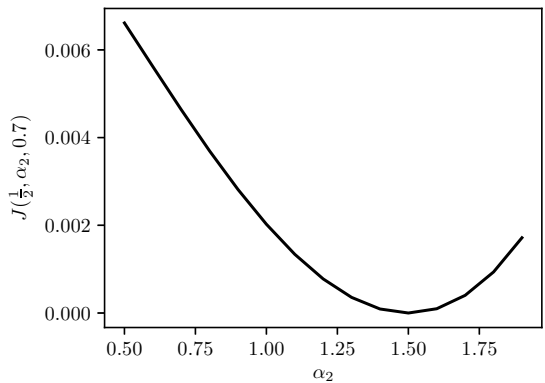}} \ \
\subfloat[Cross section $J\left(\alpha_1, \frac{3}{2}, 0.7\right)$]{\includegraphics[scale=0.75]{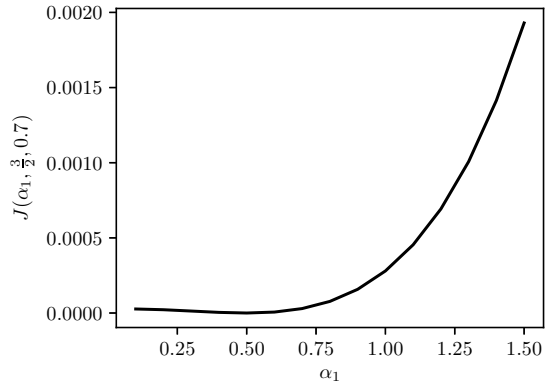}} 
\caption{Cross sections from the contour plot in Figure \ref{fig:ex01_contour_alphas}}
\label{fig:ex01_cross_sections}
\end{figure}

To compare the trajectories within the flat $(\alpha_1, \alpha_2)$-strip containing the minimizer $(\alpha_1^*,\alpha_2^*)$, we plotted regions of $(\alpha_1,\alpha_2)$-points in which the cost functional $J$ has values below various thresholds, as well as the associated trajectory deviations $u(t,0;\theta)-g(t)$. While there is only one trajectory with a deviation of $0$ (the optimal one), the trajectories associated with the parameters that lie within strips of various widths surrounding $(\alpha_1^*,\alpha_2^*)$ yield uniformly low cost function values.   

\begin{figure}[th!]
\centering
\subfloat[]{\includegraphics[scale=0.75]{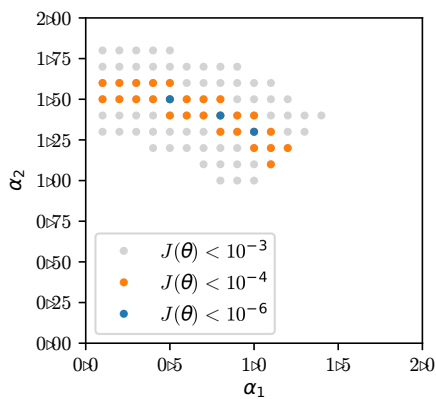}} \ \
\subfloat[]{\includegraphics[scale=0.75]{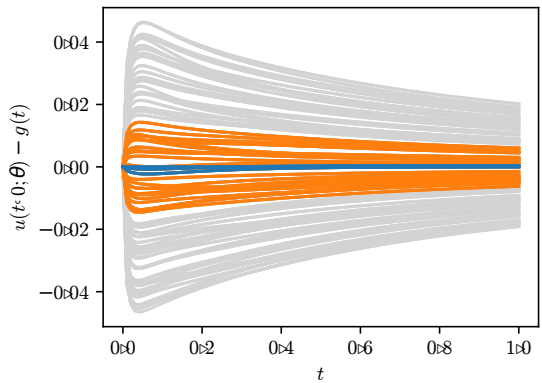}} \ \
\caption{Sets of $(\alpha_1,\alpha_2)$-pairs and associated trajectories for which the cost functional lies below thresholds $10^{-3}, 10^{-4}$, and $10^{-6}$.}
\label{fig:ex01_thresholds}
\end{figure}

Next we investigate the sensitivity of $J$ with respect to $\beta$. Since the dependence of $J$ on $\beta$ is mediated by $\alpha_1$ and $\alpha_2$, we consider the mappings $\beta\mapsto J(\alpha_1,\alpha_2,\beta)$ for a set of $(\alpha_1,\alpha_2)$-pairs that satisfy $J(\alpha_1,\alpha_2, \beta^*)<10^{-4}$, i.e. the orange points in Figure \ref{fig:ex01_thresholds} (a). Our results in Figure \ref{fig:ex01_cost_betas} show that, while there is some variation in the cost functional for different values of $(\alpha_1,\alpha_2)$, the cost functional (i) does not exhibit flat regions, and (ii) attains its minimimum consistently at $\beta=\beta^*=0.7$. 

\begin{figure}[th!]
\centering 
\includegraphics[scale=0.75]{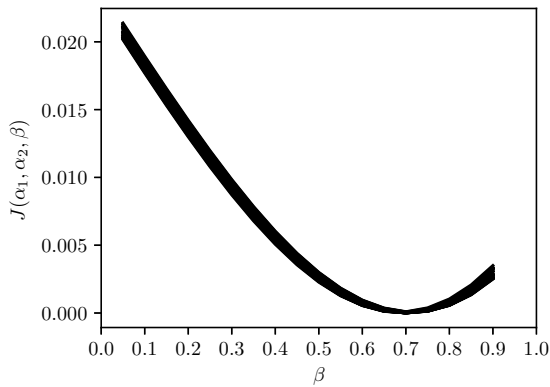}
\caption{The cost functional $\beta \mapsto J(\alpha_1,\alpha_2,\beta)$ for a selected set of $(\alpha_1,\alpha_2)$-pairs.}
\label{fig:ex01_cost_betas}
\end{figure} 

It would therefore seem that the determination of $\beta$ from trajectory data $g(t)$ is less ill-posed than that of $\alpha_1$ and $\alpha_2$ - our numerical experience supports this conjecture. 

We have shown in this section how the insensitivity of the least-squares cost functional $J$ to $\alpha_1$ and $\alpha_2$ can lead to problems in the practical identification of these parameters from an observed trajectory $g(t)$. Possible strategies for improving the identifiability of the double fractional Laplacian include: estimating $\beta$ before $\alpha_1$ and $\alpha_2$, and/or weighting the initial, transient time interval more heavily.  

{
\section{Conclusion}
\noindent We have studied a nonlocal inverse problem for the space-time fractional diffusion equation 

$\partial_t^\beta u(t, x) = -(-\Delta)^{\alpha_1/2}u(t,x) - (-\Delta)^{\alpha_2/2}u(t,x).$ We showed that the equation has a weak solution and that given an additional data $g(t)= u(t,0)$, we can uniquely determine the time and space fractional exponents $\beta, \alpha_1$ and $\alpha_2$. Finally, we provided a numerical approximation to the solution. The ill-posedness of the problem made this numerical approximation very challenging as expected. We would like to point out that we have also investigated the problem above for a spectrally defined mixed fractional Laplacian in \cite{GueNanOy}.  We also think that our methods can be used to solve a broader range of inverse problems, including the fractional diffusion equation with a source term: 

 $\partial_t^\beta u(t, x) = -(-\Delta)^{\alpha_1/2}u(t,x) - (-\Delta)^{\alpha_2/2}u(t,x)+ F(t,x)$ and distributed-order fractional diffusions.  These will be our future project.


}

\Addresses

\end{document}